\documentclass[10pt]{amsart}
\usepackage{latexsym, amsmath,amssymb}

\usepackage{hyperref}

\newcommand{\cal}{\mathcal}

\renewcommand{\epsilon}{\varepsilon}
\newcommand{\newsection}[1]
{\subsection{#1}\setcounter{theorem}{0} \setcounter{equation}{0}
\par\noindent}

\newtheorem{theorem}{Theorem}

\newtheorem{lemma}[theorem]{Lemma}
\newtheorem{corr}[theorem]{Corollary}

\newtheorem{proposition}[theorem]{Proposition}
\newtheorem{deff}[theorem]{Definition}

\newcommand{\bth}{\begin{theorem}}
\newcommand{\ble}{\begin{lemma}}
\newcommand{\bcor}{\begin{corr}}

\newcommand{\bdeff}{\begin{deff}}

\newcommand{\bprop}{\begin{proposition}}
\newcommand{\ele}{\end{lemma}}
\newcommand{\ecor}{\end{corr}}
\newcommand{\edeff}{\end{deff}}

\newcommand{\eprop}{\end{proposition}}

\newcommand{\cd}{\, \cdot\, }

\newcommand{\la}{\lambda}

\newcommand{\e}{\varepsilon}

\renewcommand{\Pi}{\varPi}

\renewcommand{\epsilon}{\varepsilon}

\newcommand{\R}{{\mathbb R}}

\newcommand{\1}{{\rm 1\hspace*{-0.4ex}%
\rule{0.1ex}{1.52ex}\hspace*{0.2ex}}}

\newcommand{\tb}{\widetilde \beta}

\begin{document}

\title[Quasimode and Strichartz estimates  for   Schr\"odinger equations]
{Quasimode and Strichartz estimates for  time-dependent Schr\"odinger equations
 with singular potentials}
%\thanks{The authors were supported in part by the NSF}
%
%
%
%
%
%

\thanks{The authors were supported in part by the NSF (NSF Grant DMS-1953413), and the second author was also partially supported by
   the Simons Foundation. }

\author[]{Xiaoqi Huang}
\address[X.H.]{Department of Mathematics,  Johns Hopkins University,
Baltimore, MD 21218}
\email{xhuang49@math.jhu.edu}

\author[]{Christopher D. Sogge}
\address[C.D.S.]{Department of Mathematics,  Johns Hopkins University,
Baltimore, MD 21218}
\email{sogge@jhu.edu}

\begin{abstract}
We generalize the Strichartz estimates for Schr\"odinger operators on compact manifolds of Burq, G\'erard and Tzvetkov~\cite{bgtmanifold} by allowing critically singular potentials $V$.  Specifically, we show that their $1/p$--loss  $L^p_tL^q_x(I\times M)$-Strichartz estimates hold for $e^{-itH_V}$ 
when $H_V=-\Delta_g+V(x)$ with $V\in L^{n/2}(M)$ if $n\ge3$ or $V\in L^{1+\delta}(M)$, $\delta>0$, if $n=2$, with $(p,q)$ being as in the
Keel-Tao theorem and $I\subset \R$ a bounded interval.  We do this by formulating and proving new ``quasimode'' estimates for scaled dyadic unperturbed Schr\"odinger operators and
taking advantage of the the fact that $1/q'-1/q=2/n$ for the endpoint Strichartz estimates when $(p,q)=(2,2n/(n-2)$.  We also show that the universal quasimode estimates that we obtain are saturated on {\em any} compact manifolds; however,  we suggest that they may lend themselves to improved
Strichartz estimates in certain geometries using recently developed ``Kakeya-Nikodym" techniques developed to obtain improved
eigenfunction estimates assuming, say, negative curvatures.
 \end{abstract}

\maketitle
\setcounter{secnumdepth}{3}

%\cite{FS}, \cite{BHSS}, \cite{BSS}, \cite{HuS}, \cite{bgtmanifold}, \cite{SFIO2}, \cite{SoggeHangzhou}, \cite{steinbeijing}, 
%\cite{KSchrod}, \cite{sogge88}, \cite{KRS}, \cite{KT}, \cite{KTZ}, \cite{SeegerSogge89}, \cite{StaffT}, \cite{SoggeKaknik},
%\cite{BourgainDemeterDecouple}, \cite{sogge2015improved}, \cite{SBLog}, \cite{BlairSoggeToponogov}, \cite{jss}, \cite{sogge86},
%\cite{LindbladSogge}, \cite{FS}, \cite{HormanderFLP}

\newsection{Introduction and main results}

In \cite{bgtmanifold}, Burq, G\'erard and Tzvetkov showed that if $(M,g)$ is an $n\ge2$ dimensional compact manifold then the time-dependent Schr\"odinger operators associated with the
Laplace-Beltrami operator satisfy
\begin{equation}\label{i.1}
\|e^{-it\Delta_g}\|_{H^{1/p}(M)\to L^p_tL^q_x(I\times M)}
\le C_I,
\end{equation}
if $I\subset \R$ is a compact interval and
\begin{equation}\label{i.2}
n(1/2-1/q)=2/p \, \, \text{and } 2\le p\le \infty\, \,
\, \text{if } \, n\ge 3, \, \, \text{or } \, 
2<p\le \infty \, \, \text{if } \, n=2.
\end{equation}
Here $H^{\sigma}(M)$ denotes the $L^2$-Sobolev space associated with the Laplace-Beltrami operator on $M$
with norm
\begin{equation}\label{i.3}
\|u\|_{H^\sigma(M)}=\bigl\| \bigl(\sqrt{I-\Delta_g}\bigr)^\sigma u\bigr\|_{L^2(M)}.
\end{equation}
In the two-dimensional case, the bounds in \eqref{i.1} also depend on $(p,q)$.  In practice there one just
takes $I=[0,1]$ since this inequality implies the bound for all compact intervals.

The main purpose of this paper is to show that we
also have the bounds in \eqref{i.1} if $-\Delta_g$
is replaced by $-\Delta_g+V(x)$, with the potential
$V$ being real-valued and
satisfying
\begin{equation}\label{i.4}
V\in L^{n/2}(M) \, \, \text{when } \, 
n\ge3  \, \, \, \text{and } 
V\in L^{1+\delta}(M) \, \, \,
\text{some } \, \,  \delta>0 \, \, \text{if } \, \,
n=2.
\end{equation}
  Such a result involves critically singular
potentials, since multiplication by elements of $L^{n/2}$ scale the same as $\Delta_g$.  Indeed, if we consider the Euclidean
Laplacian, then $\Delta u(\la\, \cdot \, )=\la^2\bigl(\Delta u\bigr)(\la\, \cdot \, )$ and $\la^2 \|V(\la \, \cdot \, )\|_{L^{n/2}}=\|V\|_{L^{n/2}}$,
and similar formulae hold on $(M,g)$ if we scale the metric.

We should also point
out that the natural $L^1\to L^\infty$ estimates
for solutions of the heat equation
involving the operators
\begin{equation}\label{i.5}
H_V=-\Delta_g+V
\end{equation}
may break down when one merely assumes that
$V\in L^{n/2}(M)$.  Moreover, individual
eigenfunctions need not be bounded (unlike the case where $V$ is smooth).  See, e.g., \cite{AZ}, \cite{BSS} and
\cite{Simonsurvey}.  
On the other hand, if $V$ is as
above then $H_V$ defines a self-adjoint operator which is bounded from below
(see, e.g., \cite{BHSS}).   Among other things, this
allows us to define the time-dependent Schr\"odinger
operators $e^{-itH_V}$.

Even though heat equation bounds may break down for
$L^{n/2}$ potentials, we do have the analog of the
Strichartz estimates \eqref{i.1} of Burq, G\'erard
and Tzvetkov:

\begin{theorem}\label{hvthm}  Let the potential $V$
be real-valued and satisfy \eqref{i.4}.  Also,
assume that the pair of exponents $(p,q)$ is as
in \eqref{i.2}.  We then have for any compact
interval $I\subset \R$
\begin{equation}\label{i.6}
\bigl\|e^{-itH_V}u\bigr\|_{L^p_tL^q_x(I\times M)}
\lesssim \|u\|_{H^{1/p}(M)}.
\end{equation}
\end{theorem}

We should point out that,  Burq, G\'erard and Tzvetkov~\cite[Theorem 6]{bgtmanifold} discussed a variant of the above theorem for the Euclidean spaces with variable coefficient metrics, and their arguments can easily be adapted to the setting of compact manifolds which would show that the above results
hold when $V\in L^n(M)$.  Also, our main point of departure from the analysis in \cite{bgtmanifold} is to use dyadic cut-offs
in the time variable as opposed to the spatial variable.  We need to do this since Littlewood-Paley operators associated
with $-\Delta_g$ are not easily seen to be compatible with ones associated to $-\Delta_g+V$ if  $V$ is singular.  We also 
note  that the philosophy that, for solutions of dispersive equations,
dyadic time-frequency cut-offs and spatial ones should be interchangeable is not new.  For instance, for solutions of 
Schr\"odinger equations this is crucially used in \cite{IvD} and \cite{PTV} and for wave equations in \cite{MSSA}.

Just as Burq, G\'erard and Tzvetkov~\cite{bgtmanifold} did for the $V\equiv 0$ case, we shall prove this result
by showing that if one restricts to frequencies comparable to $\lambda$, with $\lambda$ large one has no-loss
estimates on small intervals of size $\la^{-1}$.   
Specifically,  if fix a real-valued Littlewood-Paley
bump function
\begin{equation}\label{i.7}\beta\in C^\infty_0((1/2,2)),
\end{equation}
for future convenience, satisfying 
\begin{equation}\label{i.8}1=\sum_{-\infty}^\infty \beta(2^{-j}s)\, \, \text{for } \, \, 
s>0, \quad \text{and } \, \beta(s)=1, \, \, s\in [3/4,5/4],
\end{equation}
then the main estimate in \cite{bgtmanifold} is that for large $\la$ we have
\begin{equation}\label{i.9}
\|\beta(P/\la)e^{it\Delta_g}\|_{L^2(M)\to L^p_tL^q_x([0,\la^{-1}]\times M)} =O(1), \quad P=\sqrt{-\Delta_g},
\end{equation}
if $(p,q)$ are as in \eqref{i.2}.
Since $e^{-it\Delta_g}$ is a unitary operator on $L^2(M)$, this of course says that one has $O(1)$ bounds
on all intervals of length $\la^{-1}$, and so by adding up $O(\la)$ of these bounds they obtained the estimate
\begin{equation}\label{i.9'}
\tag{1.9$'$}
\|\beta(P/\la)e^{it\Delta_g}\|_{L^2(M)\to L^p_tL^q_x([0,1]\times M)} =O(\la^{1/p}),
\end{equation}
which leads to \eqref{i.6} with $V=0$ using standard Littlewood-Paley estimates associated with $-\Delta_g$.

We shall follow this strategy and ultimately prove analogous dyadic estimates for $e^{-itH_V}$ that will allow us
to obtain \eqref{i.6}.  We shall have to show that the Littlewood-Paley estimates for $H_V$ are valid for the
exponents $q$ in \eqref{i.2}, which we shall obtain in an appendix using a general spectral multiplier
theorem of Blunck~\cite{Blunck} and recent estimates in our collaboration with Blair and Sire~\cite{BHSS}.

In order to obtain these natural dyadic variants of \eqref{i.6} we shall rely on certain microlocalized ``quasimode'' 
estimates for the unperturbed
scaled Schr\"odinger operators with a damping term,
\begin{equation}\label{i.10}
 i\lambda\partial_t +\Delta_g+i\la.
\end{equation}
Since there is no reason to expect that the Littlewood-Paley operators associated with $-\Delta_g$
are compatible with the corresponding ones for $H_V=-\Delta_g+V(x)$ with $V$ singular, it does not
seem that we would be able to use  quasimode estimates for the unperturbed operator
$-\Delta_g$  to prove results for $H_V$ if these estimates include ``spatial'' dyadic cutoffs $\beta(P/\la)$ as above.
We shall mitigate this potential issue by using
the Littlewood-Paley operators acting on the time
variable, 
$$\beta(-D_t/\la)h(x)=(2\pi)^{-1}
\int_{-\infty}^\infty e^{it\tau}
\beta(-\tau/\la) \, \Hat h(\tau)\, d\tau,
$$
with $\beta$ as above.

Let us be more specific. Our main estimates will concern solutions of the scaled inhomogeneous 
Schr\"odinger equation with damping term
\begin{equation}\label{i.11}
(i\la \partial_t+\Delta_g+i\la)w(t,x)=F(t,x), \quad
w(0,\cd)=0.
\end{equation}
It will be convenient to assume that the ``forcing term'' here satisfies
\begin{equation}\label{i.12}
F(t,x)=0, \quad t\notin [0,1].
\end{equation}

The result that we shall use to prove Theorem~\ref{hvthm} then is the following.

\begin{theorem}\label{mainthm}  Suppose that
$F$ satisfies the support assumption in \eqref{i.12}
and that $w$ solves \eqref{i.11}.  Then for $\la\ge 1$ and
exponents as in \eqref{i.2} we have
\begin{equation}\label{i.13}
\bigl\|\beta(-D_t/\la)w\bigr\|_{L^p_tL^q_x(\R\times M)}
\lesssim \la^{-1+1/p}\|F\|_{L^2_{t,x}([0,1]\times M)},
\end{equation}
and also
\begin{equation}\label{i.14}
\bigl\|\beta(-D_t/\la)w\bigr\|_{L^p_tL^q_x(\R\times M)}
\lesssim \la^{-1+2/p}
\|F\|_{L^{p'}_tL^{q'}_x([0,1]\times M)}.
\end{equation}
Furthermore, the quasimode estimates \eqref{i.13} are sharp on {\em any} manifold.
\end{theorem}

These ``microlocalized quasimode estimates'' 
are  natural analogs of the ones obtained
by one of us for Laplace-Beltrami operators in \cite{sogge88}.  
Additionally, the first estimate, \eqref{i.13}, essentially follows from
the results of one of us and Seeger~\cite{SeegerSogge89}.
As was the case in these earlier works, and more recently in 
  \cite{BSS} and \cite{BHSS}, it is natural to include the ``damping
term'',  $i\la$, to exploit the Fourier analysis that arises.  In the present
context, it allows use Fourier analysis 
in $\R$ to link dyadic microlocal cutoffs in the spatial variable involving $P=\sqrt{-\Delta_g}$ with the above
ones involving
the time variable.  As we shall see, being able to prove time-microlocalized estimates for solutions of
inhomogeneous equations involving the operators in \eqref{i.10} will allow us to use the Duhamel
formula to prove our estimates for $e^{-itH_V}$ in a manner that is somewhat reminiscent to arguments
in a recent joint work \cite{BHSS} on uniform Sobolev estimates for the operators $H_V$.  It is for this
reason, and others, that it is important for us to prove natural estimates for inhomogeneous equations, as
opposed to ones just involving the Cauchy problem.  On the other hand, our proof of 
\eqref{i.13} and \eqref{i.14} will be modeled by the arguments in \cite{bgtmanifold} that lead to \eqref{i.9}.
In \S~\ref{sharpsec} we shall show that the quasimode estimates \eqref{i.13} are optimal.

This paper is organized as follows.  In the next section, we shall prove Theorem~\ref{mainthm}.  Then, in
\S~\ref{stsec} we shall show how we can use the above Theorem along with Littlewood-Paley estimates
to obtain the Strichartz estimates in Theorem~\ref{hvthm}.  We shall prove the Littlewood-Paley estimates
that we require in an Appendix.  In \S~\ref{sharpsec} we shall show that the universal bounds
\eqref{i.13} that easily imply the estimates  \eqref{i.9'} are saturated on {\em any} manifold, even 
though Bourgain and Demeter~\cite{BourgainDemeterDecouple} showed much better estimates hold on the torus  when $p=q=2(n+2)/n$ with just an $\la^\epsilon$-loss for all $\e>0$,
 and Burq, G\'erard and Tzvetkov~\cite{bgtmanifold} also showed that on spheres there are improvements of \eqref{i.9'} 
 in many cases.  It seems a challenge to show that there are improvements in more general cases, such as when $M$
 has negative sectional curvature; however, the Knapp example that we shall construct in 
 \S~\ref{sharpsec}  suggests that perhaps the ``Kakeya-Nikodym" techniques that have been recently developed
 in \cite{BlairSoggeToponogov}, \cite{SBLog}, \cite{SoggeKaknik}, \cite{sogge2015improved} and
 \cite{SoggeZelditchL4} 
 to obtain improved eigenfunction estimates in certain geometries
 might lend themselves
 to this problem.
 
 We are grateful to the referee for several helpful suggestions which improved our exposition.

%\newsection{Strichartz estimates for Schr\"odinger operators with singular potentials in Euclidean space}

\newsection{Quasimode estimates for scaled  Schr\"odinger operators}\label{qmsect}

In this section we shall prove Theorem~\ref{mainthm}.
If $\beta$ is as in \eqref{i.7}--\eqref{i.8}, let us define
``wider cutoffs" that we shall also use as follows
\begin{equation}\label{m.2}
\tb(s)=\sum_{|j|<10}\beta(2^{-j}s)\in C^\infty_0(2^{-10},
2^{10}).
\end{equation}
For future use, note that
\begin{equation}\label{m.3}
\tb(s)=1 \quad \text{on } \, \, (1/4,4).
\end{equation}

One of the main estimates in \cite{bgtmanifold} is that one can
obtain the ``expected" $O(|t|^{-n/2})$ dispersive
estimates for $\beta(P/\la)e^{it\Delta_g}$, $P=\sqrt{-\Delta_g}$, on time intervals of the form
$[-\ell(\la),\ell(\la)]$ for $\la\gg 1$ if
$\ell(\la)=\delta \la^{-1}$ for some $\delta=\delta_M>0$.  Using the Weyl formula, they also showed that these 
$O(|t|^{-n/2})$ $L^1\to L^\infty$ bounds are optimal
in the sense that no such uniform bounds are possible
if $\la \ell(\la)\to \infty$ as $\la \to \infty$.
Using the bounds for each fixed Littlewood-Paley
bump function $\beta(2^{-j} \, \cdot \, )$, one
can of course obtain analogous $O(|t|^{-n/2})$ dispersive
estimates involving $\tb$ in \eqref{m.2} on intervals
$[-\delta \la^{-1},\delta \la^{-1}]$.  So, after 
possibly changing scales in time and correspondingly
scaling the Laplace-Beltrami operator, we may always
assume that we have the bounds
\begin{equation}\label{m.4}
\bigl\| \, \tb(P/\la)e^{it\Delta_g}\, \bigr\|_{L^1(M)
\to L^\infty(M)}\le C|t|^{-n/2}, \quad
|t|\le \la^{-1},
\end{equation}
by virtue of \cite[Lemma 2.5]{bgtmanifold}.  We also, trivially
for {\em all} times $t$ have the bounds
\begin{equation}\label{m.5}
\bigl\| \, \tb(P/\la)e^{it\Delta_g}\, \bigr\|_{L^2(M)
\to L^2(M)} \le C, \quad C=\|\tb\|_{L^\infty}.
\end{equation}

As was noted in \cite{bgtmanifold} one can use the Keel-Tao
theorem \cite[Theorem 1.2]{KT} to obtain the uniform
dyadic Strichartz estimates
\begin{multline}\label{m.6}
\bigl\|\tb(P/\la)e^{it\Delta_g} f
\bigr\|_{L^p_tL^q_x([0,\la^{-1}]\times M)}
\le C\|f\|_{L^2(M)},
\\
\text{if } \, 
\quad n(1/2-1/q)=2/p \, \, \text{and } 2\le p<\infty\, \,
\, \text{if } \, n\ge 3, \, \, \text{or } \, 
2<p<\infty \, \, \text{if } \, n=2.
\end{multline}
We have excluded the case of $p=\infty$ in \eqref{m.6} since then
$q=2$ and the estimate is trivial (with $[0,\la^{-1}]$ replaced
by any interval) by the spectral theorem.
Also, for future use, note that the endpoint case in
dimensions $n\ge3$ involves the exponents $p=2$ and
$q=2n/(n-2)$, for which we have $1/q'-1/q=2/n$, where,
as usual, $q'$ denotes the conjugate exponent.  
Being able to include this estimate will allow
us to handle potentials $V\in L^{n/2}$ when $n\ge3$,
while the fact that the endpoint estimate for $n=2$
breaks down, 
%forces us to 
accounts for the reason that we
assume that our potentials
lie in $L^{1+\delta}$, some $\delta>0$ if $n=2$.

We shall use an equivalent variant of this estimate
and the related estimate for inhomogeneous equations
that will be formulated first for the unit interval
to simplify the Fourier analysis to follow.  We first,
trivially note that \eqref{m.6} is equivalent
to the estimate
\begin{equation}\label{m.7}
\bigl\|\tb(P/\la)e^{it \la^{-1}\Delta_g}e^{-t} f
\bigr\|_{L^p_tL^q_x([0,1]\times M)}
\le C \la^{1/p}\|f\|_{L^2(M)},
\, \, (p,q) \, \, \text{as in }\, \eqref{m.6},
\end{equation}
and since $e^{is\Delta_g}$ has $L^2(M)\to L^2(M)$
operator norm one, we also have the damped ``global''
estimate
\begin{equation}\label{m.8}
\bigl\|\tb(P/\la)e^{it\la^{-1}\Delta_g}e^{-t} f
\bigr\|_{L^p_tL^q_x([0,+\infty)\times M)}
\le C \la^{1/p} \|f\|_{L^2(M)},
\, \, (p,q) \, \, \text{as in }\, \eqref{m.6}.
\end{equation}
Note that for the scaled Schr\"odinger operator in \eqref{i.10} we
have
\begin{equation}\label{m.9}
\bigl(i\la \partial_t+\Delta_g+i\la)\bigl(e^{it\la^{-1}\Delta_g} e^{-t} h)(x)=0.
\end{equation}

To proceed, let $\1_+(s)=\1_{[0,+\infty)}(s)$ denote the Heaviside function and
\begin{equation}\label{m.10}
U(t)=\1_+(t) \tb(P/\la) e^{it\la^{-1}\Delta_g}e^{-t}
\end{equation}
be the operator in \eqref{m.8}.  For later use,
let us note that we can rewrite this operator.
Indeed,
if we recall that
\begin{equation*}
(2\pi)^{-1}\int_{-\infty}^\infty \frac{e^{it\tau}}{i\tau +1} \, d\tau = \1_+(t)e^{-t},
\end{equation*}
we deduce that  
\begin{equation}\label{m.11}
U(t)f(x)=\frac{i\la}{2\pi}
\int_{-\infty}^\infty 
\frac{e^{it\tau}}{-\la\tau + \Delta_g+i\la}
\, \tb(P/\la)f(x) \, d\tau.
\end{equation}
Also, if we regard $U$ as an operator sending functions
of $x$ into functions of $x,t$, then its adjoint is the operator
\begin{equation}\label{m.12}
U^*F(x)=\int_0^\infty e^{-s}\bigl(e^{-is\la^{-1}\Delta_g}
\tb(P/\la) F(s,\cd )\bigr)(x) \, ds.
\end{equation}
Consequently,
\begin{multline}\label{m.13}
\int U(t)U^*(s)F(s,x)\, ds
\\
=\1_+(t) \int_0^\infty
\Bigl(e^{i(t-s)\la^{-1}\Delta_g}e^{-(t-s)}
\tb^2(P/\la) e^{-2s}F(s,\cd)\Bigr)(x) \, ds.
\end{multline}
Note also that if, say,
\begin{equation}\label{m.14}
F(t,x)=0, \quad t\notin [0,1],
\end{equation}
then
 the solution to the scaled
inhomogeneous Schr\"odinger equation with damping term
\begin{equation}\label{m.15}
(i\la \partial_t+\Delta_g+i\la)w(t,x)=F(t,x), \quad
w(0,\cd)=0
\end{equation}
is given by
\begin{multline}\label{m.16}
w(t,x)=(i\la)^{-1}\int_0^t 
 \bigl(\Bigl(e^{i(t-s)\la^{-1}\Delta_g}e^{-(t-s)}
F(s,x)\bigr) \, ds
\\
=(2\pi)^{-1}\int_0^1\int_{-\infty}^\infty
\frac{e^{i(t-s)\tau}}{-\la\tau+\Delta_g+i\la}
F(s,\cd)(x)\, d\tau ds.
\end{multline}
Thus, since $w(t,\cd)=0$ for $t<0$, it follows
from \eqref{m.10}, \eqref{m.11} and \eqref{m.16}
that  
\begin{multline}\label{m.17}
\tb^2(P/\la)w(t,x) =(i\la)^{-1}\int_0^t
U(t)U^*(s) \bigl(e^{2s}\tb^2(P/\la)F(s,\cd)\bigr)(x) \, ds
\\
= (2\pi)^{-1}\int_0^1\int_{-\infty}^\infty
\frac{e^{i(t-s)\tau}}{-\la\tau+\Delta_g+i\la}
\tb^2(P/\la)F(s,\cd)(x)\, d\tau ds.
\end{multline}

Using these formulas, we claim that we can use the
arguments of Burq, G\'erard and Tzvetkov~\cite{bgtmanifold}
along with the Keel-Tao \cite{KT} theorem to
deduce the following.

\begin{proposition}\label{mainprop}  Suppose that
$F$ satisfies the support assumption in \eqref{m.14}
and that $w$ solves \eqref{m.15}.  Then for
exponents as in \eqref{m.6} and $\la\ge1$ we have
\begin{equation}\label{m.18}
\bigl\|\tb^2(P/\la)w\bigr\|_{L^p_tL^q_x(\R\times M)}
\lesssim \la^{-1+1/p}\|F\|_{L^2_{t,x}([0,1]\times M)},
\end{equation}
and also
\begin{equation}\label{m.19}
\bigl\|\tb^2(P/\la)w\bigr\|_{L^p_tL^q_x(\R\times M)}
\lesssim \la^{-1+2/p}
\|F\|_{L^{p'}_tL^{q'}_x([0,1]\times M)}.
\end{equation}
\end{proposition}

We remark that, like \eqref{i.13}, the bounds in \eqref{m.18} are also optimal.

\begin{proof}
To use the dispersive estimates \eqref{m.4} of Burq, G\'erard and 
Tzvetkov, let
\begin{equation}\label{m.20}
V(t')f(x)= \1_{[0,\la^{-1}]}(t')U(\la t')f(x)= \1_{[0,\la^{-1}]}(t') e^{-\la t'} \tb(P/\la) e^{it'\Delta_g}f(x).
\end{equation}
We then clearly have
$$\|V(t')\|_{L^2(M)\to L^2(M)}=0(1),$$
and \eqref{m.4} says that
$$\|V(t')(V(s'))^*\|_{L^1(M)\to L^\infty(M)}
\lesssim |t'-s'|^{-n/2}.$$

We can use the Keel-Tao theorem along with these
two inequalities to deduce 
that
$$\|V(t')f\|_{L^p_{t'}L^q_x([0,\la^{-1}]\times M)}
\lesssim \|f\|_{L^2(M)},
$$
as well
as 
$$\Bigl\|\int_0^{t'}V(t')V^*(s')G(s',\cd)\, ds'
\, \Bigr\|_{L^p_{t'}L^q_x([0,\la^{-1}]\times M)}
\lesssim \|G\|_{L^{p'}_{t}L^{q'}_x([0,\la^{-1}]\times
M)},
$$
and
$$\Bigl\| \int_0^{\la^{-1}}V^*(s')G(s', \cd)\, ds'
\Bigr\|_{L^2(M)}
\lesssim  \|G\|_{L^{p'}_tL^{q'}_x([0,\la^{-1}]\times
M)}.
$$
Using \eqref{m.20} we deduce that these inequalities
are equivalent to
\begin{equation}\label{m.21}
\|U(t)f\|_{L^p_tL^q_x([0,1]\times M)}
\lesssim \la^{1/p}\|f\|_{L^2(M)},
\end{equation}
as well as
\begin{multline}\label{m.22}
\Bigl\|\int_0^t U(t)U^*(s)H(s,\cd)\, ds
\Bigr\|_{L^p_tL^q_x([0,1]\times M)}
\\
\lesssim \la^{2/p} \,
\|H\|_{L^{p'}_tL^{q'}_x([0,1]\times
M)}, \quad \text{if } \, \, H(s,\cd)=0, \, \, s\notin
[0,1],
\end{multline}
and
\begin{equation}\label{m.23}
\Bigl\|\int_0^1 U^*(s)H(s,\cd)\, ds\Bigr\|_{L^2(M)}
\lesssim \la^{1/p}\|H\|_{L^{p'}_tL^{q'}_x([0,1]\times
M)},
\end{equation}
respectively.

Using \eqref{m.22} with $H=e^{2s}F$ along with \eqref{m.17} we obtain the analog of \eqref{m.19}
where the norms are taken over $[0,1]\times M$
since $\|H\|_{L^p_tL^q_x}\approx \|F\|_{L^p_tL^q_x}$
due to \eqref{m.14}.  Since, as we noted before
$w$ and hence $\tb^2(P/\la)w$ vanishes for $t<0$, to
prove the remaining part of \eqref{m.22} we need
that we also have
\begin{equation}\label{m.24}
\bigl\|\tb^2(P/\la)w\bigr\|_{L^p_tL^q_x([1,\infty)\times M)}
\lesssim \la^{-1+2/p}
\|F\|_{L^{p'}_tL^{q'}_x([0,1]\times M)}.
\end{equation}
Since for $t>1$
$$\int_0^t U(t)U^*(s)H(s,\cd)\, ds
=U(t)\Bigl(\int_0^1 U^*(s)H(s,\cd)\, ds\Bigr),
\quad H=e^{2s}F,
$$
it is simple to check that by \eqref{m.23} we would
have this inequality if 
\begin{equation}\label{m.25}\|U(t)f\|_{L^p_tL^q_x([1,\infty)\times M)}
\lesssim \la^{1/p}\|f\|_{L^2(M)}.
\end{equation}
However, since for $j=1,2,\dots$ and $s\in [0,1]$
$$U(j+s)= e^{-j}e^{i\la^{-1}j\Delta_g}U(s),$$
and $\| e^{it\la^{-1}\Delta_g}\|_{L^2\to L^2}=1$, we deduce that 
$$\bigl\|\tb^2(P/\la)w\bigr\|_{L^p_tL^q_x([j,j+1]\times M)}
\lesssim e^{-j}\la^{-1+2/p}
\|F\|_{L^{p'}_tL^{q'}_x([0,1]\times M)}, \, \,
j=1,2,3,\dots,
$$
which of course yields \eqref{m.24}, and, as a result,
\eqref{m.18}.

Since $\|U^*(s)\|_{L^2(M)\to L^2(M)}=O(1)$,
using \eqref{m.17} along with \eqref{m.21}
and
\eqref{m.25} we find that if $H=e^{2s}F$
\begin{align*}
\bigl\|\tb^2(P/\la)w\bigr\|_{L^p_tL^q_x(\R\times M)}
&\le \la^{-1} \int_0^1
\bigl\|\1_+(t-s) U(t)U^*(s)H(s, \cd)
\bigr\|_{L^p_tL^q_x(\R\times M)}\, ds
\\
&\lesssim
\la^{-1+1/p}\int_0^1\|U^*(s)H(s, \cd)\|_{L^2_x}\, ds
\\
&\lesssim \la^{-1+1/p}\int_0^1\|F(s,\cd)\|_{L^2_x}\, ds
\le \la^{-1+1/p}\|F\|_{L^2_{t,x}([0,1]\times M)},
\end{align*}
as desired, which completes the proof.\end{proof}

By an argument we shall give in the next section
the quasimode estimates \eqref{m.18} for the
scaled Schr\"odinger operators in \eqref{i.10} imply
the dyadic Strichartz estimates \eqref{m.4}
of Burq, G\'erard and Tzvetkov~\cite{bgtmanifold}.  Unfortunately, though, as we noted before, we do not seem to be able to
directly use Proposition~\ref{mainprop} to obtain
analogous estimates for $-H_V=\Delta_g-V$ with
$V\in L^{n/2}(M)$, $n\ge 3$ or the 2-dimensional ones in
Theorem~\ref{hvthm}, since Littlewood-Paley
operators associated with $H_V$ are not easily seen
to be compatible with the corresponding ones 
involving $-\Delta_g$ if $V$ is allowed to be singular.

It is for this reason that we need the bounds in Theorem~\ref{mainthm} involving the
Littlewood-Paley cutoff $\beta(-D_t/\la)$ in the time-variable.
We are now in a position to prove this result.  We shall use Proposition~\ref{mainprop}
and the following two elementary lemmas whose proofs we postpone for the moment.

%\begin{theorem}\label{mainthm}  Suppose that
%$F$ satisfies the support assumption in \eqref{m.14}
%and that $w$ solves \eqref{m.15}.  Then for $\la\ge 1$ and
%exponents as in \eqref{m.6} we have
%\begin{equation}\label{i.13}
%\bigl\|\beta(-D_t/\la)w\bigr\|_{L^p_tL^q_x(\R\times M)}
%\lesssim \la^{-1+1/p}\|F\|_{L^2_{t,x}([0,1]\times M)},
%\end{equation}
%and also
%\begin{equation}\label{i.14}
%\bigl\|\beta(-D_t/\la)w\bigr\|_{L^p_tL^q_x(\R\times M)}
%\lesssim \la^{-1+2/p}
%\|F\|_{L^{p'}_tL^{q'}_x([0,1]\times M)}.
%\end{equation}
%\end{theorem}

%To prove this result we shall use Proposition~\ref{mainprop} and the following two elementary lemmas whose proofs we postpone for the moment.

\begin{lemma}\label{soblemma}  Let $\alpha\in C([0,\infty))$ and $1<p\le 2<q<\infty$.  Then
\begin{equation}\label{m.28}
\|\alpha(P)f\|_{L^q(M)}
\le C_{p,q} \, 
\bigl(\sup_{\mu\ge 0} (1+\mu)^{n(\frac1p-\frac1q)}
|\alpha(\mu)|\bigr) \, \|f\|_{L^p(M)}.
\end{equation}
\end{lemma}

\begin{lemma}\label{freeze}  Suppose that
\begin{equation}\label{m.29}
|K_\la(t,t')|\le \la (1+\la|t-t'|)^{-2}.
\end{equation}
Then if $1\le p\le q\le \infty$ we have
the following uniform bounds for $\la\ge1$
\begin{equation}\label{m.30}
\Bigl\|\, \int K_\la(t,t') \, 
G(t',\cd)\, dt' \, 
\Bigr\|_{L^p_tL^q_x(\R\times M)}
\le C\|G\|_{L^p_tL^q_x(\R\times M)}.
\end{equation}
Also, suppose that 
$$WF(t,x)=
\int_{-\infty}^\infty \int_M
K(t,x;t',y) \, F(t',y) \, d\mathrm{Vol}(y) \, ds'
$$
and that for each $t,t'\in \R$ the operator
$$W_{t,t'}f(x) = \int_M K(x,t;y,t') f(y) \, d\mathrm{Vol}(y)$$
satisfies
$$\|W_{t,t'}f\|_{L^q(M)}
\le \la(1+\la|t-t'|)^{-2} \, \|f\|_{L^r(M)}
$$
for some $1\le r\le q\le \infty$.  Then if
$1\le s\le p\le \infty$ we have for $\la\ge1$
\begin{equation}\label{m.31}
\|WF\|_{L^p_tL^q_x(\R\times M)}
\le C\la^{\frac1s-\frac1p}
\|F\|_{L^s_tL^r_x(\R\times M)}.
\end{equation}
\end{lemma}

\begin{proof}[Proof of Theorem~\ref{mainthm}]
We first note that the kernel of $\beta(-D_t/\la)$
is $O(\la (1+\la|t-t'|)^{-2})$.  Therefore, by
\eqref{m.30}
$$\|\beta(-D_t/\la) \tb^2(P/\la)w\|_{L^p_tL^q_x(\R\times M)}
\lesssim \| \tb^2(P/\la)w\|_{L^p_tL^q_x(\R\times M)}.
$$
Therefore, if as in Proposition~\ref{mainprop} and
our theorem our forcing term $F$ satisfies \eqref{m.14},
it suffices to show that $\beta(-D_t/\la)(I-\tb^2(P/\la))w$ enjoys the bounds in 
\eqref{i.13} and \eqref{i.14}.

Recalling \eqref{m.16}, this means that it suffices
to show that
\begin{multline}\label{m.32}
\Bigl\|
\int_0^1\int_{-\infty}^\infty
\frac{e^{i(t-s)\tau}}{-\la\tau-P^2+i\la}
\beta(-\tau/\la) \, 
\bigl(1-\tb^2(P/\la)\bigr)\, F(s,\cd)\bigr)\, d\tau ds
\, \Bigr\|_{L^p_tL^q_x(\R\times M)}
\\
\lesssim \la^{-1+1/p}\|F\|_{L^2_{t,x}([0,1]\times M)},
\end{multline}
as well as
\begin{multline}\label{m.33}
\Bigl\|
\int_0^1\int_{-\infty}^\infty
\frac{e^{i(t-s)\tau}}{-\la\tau-P^2+i\la}
\beta(-\tau/\la) \, 
\bigl(1-\tb^2(P/\la)\bigr)\, F(s,\cd)\bigr)\, d\tau ds
\, \Bigr\|_{L^p_tL^q_x(\R\times M)}
\\
\lesssim \la^{-1+2/p}
\|F\|_{L^{p'}_tL^{q'}_x([0,1]\times M)}.
\end{multline}

To use Lemma~\ref{freeze} set
$$\alpha(t,s;\mu)
= \int_{-\infty}^\infty
\frac{e^{i(t-s)\tau}}{-\la\tau-\mu^2+i\la}
\beta(-\tau/\la) \, 
\bigl(1-\tb^2(\mu/\la)\bigr) \, d\tau,$$
and note that, by \eqref{m.3} and the support
properties of $\beta$ we have
for $j=0,1,2$
$$
\la \,
\bigl|
\la^{j} \,
\partial_\tau^j \bigl((1-\tb^2(\mu/\la)\bigr)  \beta(-\tau/\la)
(-\la\tau-\mu^2+i\la)^{-1}\bigr)\, \bigr|
\lesssim \la (\mu^2+\la^2)^{-1},$$
which, by a simple  integration parts argument, translates to the bound
$$|\alpha(t,s;\mu)|\lesssim 
\la(1+\la|t-s|)^{-2} \cdot (\mu^2+\la^2)^{-1}.$$

If we use Lemma~\ref{soblemma} we deduce from this that the 
``frozen operators" 
$$T_{t,s}h(x)
=\int_{-\infty}^\infty
\frac{e^{i(t-s)\tau}}{-\la\tau-P^2+i\la}
\beta(-\tau/\la) \, 
\bigl(1-\tb^2(P/\la)\bigr)h(x)\, d\tau,$$
satisfy
\begin{multline}\label{m.34}
\|T_{t,s}h\|_{L^q(M)}
\lesssim \la(1+\la|t-s|)^{-2}
\cdot \la^{-2+n(1/2-1/q)}\|h\|_{L^{2}(M)}
\\
= \la(1+\la|t-s|)^{-2} \cdot \la^{-2+2/p}
\|h\|_{L^{2}(M)},
\end{multline}
as well as
\begin{multline}\label{m.35}
\|T_{t,s}h\|_{L^q(M)}
\lesssim \la(1+\la|t-s|)^{-2}
\cdot \la^{-2+n(1/q'-1/q)}\|h\|_{L^{q'}(M)}
\\
= \la(1+\la|t-s|)^{-2}
\cdot \la^{-2+4/p}\|h\|_{L^{q'}(M)},
\end{multline}
due to the fact that our assumption on the exponents
in \eqref{m.6} means that $n(1/2-1/q)=2/p$
 and $n(1/q'-1/q)=2n(1/2-1/q)=4/p$.

If we combine \eqref{m.34} and \eqref{m.31}, we conclude
that the left side of \eqref{m.32} is dominated by
$$\la^{\frac1{2}-\frac1p}
\cdot \la^{-2+2/p}\|F\|_{L^2_{t,x}(\R\times M)}
=\la^{-\frac32+\frac1p} \|F\|_{L^2_{t,x}(\R\times M)}
,
$$
which is better than the bounds posited in \eqref{m.32}
by a factor of $\la^{-1/2}$.

Similarly, if we combine \eqref{m.35} and \eqref{m.31}, we find that the left side of \eqref{m.33} is dominated
by
$$\la^{\frac1{p'}-\frac1p}\la^{-2+4/p}\|F\|_{L^{p'}_t
L^{q'}_x(\R\times M)}
=\la^{-1+2/p} \|F\|_{L^{p'}_t
L^{q'}_x(\R\times M)},
$$
as desired, which completes
the proof.
\end{proof}

To conclude this section,
for the sake of completeness let us now prove the lemmas, both of which are well known.  The first is a slight
generalization of Lemma 2.3 in \cite{BSSY}, for instance,
while the second lemma is essentially Theorem 0.3.6
in \cite{SFIO2}.

\begin{proof}[Proof of Lemma~\ref{soblemma}]
Since $(1+P)^{-n(\frac12-\frac1q)}: \, L^2(M)\to L^q(M)$
and $(1+P)^{-n(\frac1p-\frac12)}: \, L^p(M)\to
L^2(M)$, by orthogonality, we obtain
\begin{align*}
\|\alpha(P)f\|_{L^q(M)}
&\lesssim \bigl\| \, (1+P)^{n(\frac12-\frac1p)}
\alpha(P)f\, \bigr\|_{L^2(M)}
\\
&\le \bigl(\sup_{\mu\ge 0} (1+\mu)^{n(\frac1p-\frac1q)}
|\alpha(\mu)|\bigr) \cdot \|(1+P)^{-n(\frac1p-\frac12)}
f\|_{L^2(M)}
\\
&\lesssim \bigl(\sup_{\mu\ge 0} (1+\mu)^{n(\frac1p-\frac1q)}
|\alpha(\mu)|\bigr) \cdot \|f\|_{L^p(M)},
\end{align*}
as desired.
\end{proof}

\begin{proof}[Proof of Lemma~\ref{freeze}]
To obtain \eqref{m.30} we note that
by Minkowski's inequality and \eqref{m.29}
$$ \Bigl\| \int K_\la(t,t') \, 
G(t', \cd) \, dt' \,
\Bigr\|_{L^q_x(M)}
\le 
\int \la (1+\la|t-t'|)^{-2} \,
\|G(t',\cd)\|_{L^q_x(M)} \, dt'.
$$
Taking the $L^p_t$-norm of both sides
and using Young's inequality yields
\begin{multline*}
\Bigl\| \int K_\la(t,t') \, G(t', \cd) \, dt'
\Bigl\|_{L^p_tL^q_x(\R\times M)}
\le \Bigl(\int \,
\Bigl|\, 
\int \la (1+\la|t-t'|)^{-2} \,
\|G(t',\cd)\|_{L^q_x(M)} \, dt' \, \Bigr|^p \, dt
\Bigr)^{1/p}
\\
\le C
\Bigl(\, \int \, \|G(t,\cd)\|_{L^q_x(M)}^p \, dt
\, \Bigr)^{1/p}=\|G\|_{L^p_tL^q_x(\R\times M)},
\end{multline*}
as desired.

One also obtains \eqref{m.31} from this argument
after noting that, by Young's inequality,
convolution with $\la(1+\la|t|)^{-2}$ has 
$L^s(\R)\to L^p(\R)$ operator norm which
is $O(\la^{\frac1s-\frac1p})$.
\end{proof}

\newsection{%Dyadic 
Strichartz estimates on compact manifolds}\label{stsec}

Let us now see how we can use the first estimate in Theorem~\ref{mainthm} to prove the dyadic
Strichartz estimate \eqref{m.6} of Burq, G\'erard and Tzvetkov~\cite{bgtmanifold}.  This simple argument
will serve as a model for the one we shall use to prove the same sort of bounds where we replace
$-\Delta_g$ with $H_V=-\Delta_g+V$, with $V$ singular.

Let us first recall that the spectrum of $\sqrt{-\Delta_g}$ is nonnegative and discrete.  If we account
for multiplicity, we can arrange the eigenvalues,
$0=\lambda_0<\lambda_1\le \lambda_2\le \dots$ and the
associated $L^2$-normalized eigenfunctions
$$-\Delta_g e_j=\lambda_j^2, \quad
\int_M |e_j|^2\,dV_g=1$$
form an orthonormal basis for $L^2(M)$.  If then
$$E_jf(x)=\bigl(\, \int_M f \, \, \overline{e_j}\,dV_g
\, \bigr) \, e_j(x)$$
denotes the projection onto the $j$-th eigenspace we have
$$e^{it\Delta_g}f=\sum_{j=0}^\infty e^{-it\lambda_j^2}
E_jf.$$

\newcommand{\ten}{[9\la/10, 11\la/10]}

To prove \eqref{m.6} it clearly suffices to show that
for large $\lambda$ we have the uniform bounds
\begin{multline}\label{d.1}
\Bigl\| \, \eta(\la t) e^{it\Delta_g}f_\la
\Bigr\|_{L^p_tL^q_x(\R\times M)}\le C\|f_\la\|_{L^2(M)},
\\
\text{if } \, \, \mathrm{spec } \, f_\la
\subset [9\la/10, 11\la/10] \quad
\text{and } \, \, \,
\eta\in C^\infty_0((0,1)) \quad
\text{is fixed}.
\end{multline}

The assumption on the spectrum of $f_\la$ is that
$E_jf_\la=0$ if $\la_j\notin \ten$, and we choose
this interval since we are assuming that the 
Littlewood-Paley bump function arising in
Theorem~\ref{mainthm} satisfies
\begin{equation}\label{d.2}
\beta(s)=1 \quad \text{on } \, \,
[3/4,5/4] \quad
\text{and } \, \, \mathrm{supp } \,\beta
\subset (1/2,2).
\end{equation}

To be able to use \eqref{i.13} we note that, after
rescaling, \eqref{d.1} is equivalent to the
statement that
\begin{equation}\label{d.1'}\tag{3.1$'$}
\|w\|_{L^p_tL^q_x(\R\times M)}
\le C\la^{1/p} \|f_\la\|_{L^2(M)}, 
\quad \text{with } \, w(t,x)=\eta(t) \cdot e^{it\la^{-1}\Delta_g} f_\la(x).
\end{equation}

To be able to use Theorem~\ref{mainthm} we shall
use the following simple lemma.

\begin{lemma}\label{error}  Let $w$ be as in
\eqref{d.1'} with $\eta$ and $f_\la$ as in \eqref{d.1}
and suppose that the exponents $(p,q)$ are as in
\eqref{m.6}.
Then for large enough $\la$ and each $N=1,2,\dots$
we have the uniform bounds
\begin{equation}\label{d.3}
\bigl\| \, (I-\beta(-D_t/\la))w \,
\bigr\|_{L^p_tL^q_x(\R \times M)}
\le C_N \la^{-N}\|f_\la\|_{L^2(M)}.
\end{equation}
\end{lemma}

\begin{proof}  We first note the Fourier transform
of $t\to \eta(t)e^{-it\la^{-1}\la^2_j}$ is
$\Hat \eta(\tau+\la_j^2/\la)$ and so
\begin{equation}\label{d.4}
\bigl(I-\beta(-D_t/\la)\bigr)w(t,x)
= \sum_{\la_j\in \ten} a(t;\la_j) E_jf_\la(x),
\end{equation}
where
\begin{equation}\label{d.5}
a(t;\mu)=(2\pi)^{-1}
\int_{-\infty}^\infty e^{it\tau}
\Hat \eta(\tau+\mu^2/\la) \, 
\big(1-\beta(-\tau/\la)\bigr) \, d\tau.
\end{equation}

Since for $q$ as in \eqref{m.6} we have $2<q\le 2n/(n-2)$ for $n\ge 3$ and $2<q<\infty$ for $q=2$, 
the following Sobolev estimates are valid
\begin{equation}\label{d.s}
\|u\|_{L^q(M)}\lesssim \| \, (I-\Delta_g)^{1/2}u\, \|_{L^2(M)}.
\end{equation}
Therefore, by the spectral theorem,
\begin{equation}\label{d.6}
\bigl\| \sum_{\la_j\in \ten} a(t;\la_j)E_jf_\la \bigl\|_{L^q(M)}\lesssim 
\la \bigl\| \sum_{\la_j\in \ten} a(t;\la_j)E_jf_\la \bigl\|_{L^2(M)}.
\end{equation}

Next, since $2\le p<\infty$, by Minkowski's inequality and Sobolev's theorem for $\R$ we therefore have
\begin{align*}
\|(I-\beta(-D_t/\la))w\|_{L^p_tL^q_x(\R\times M)} &\lesssim 
\la \bigl\| \sum_{\la_j\in \ten} a(t;\la_j)E_jf_\la \bigl\|_{L^p_tL^2_x(\R\times M)}
\\
&\le \la \bigl\| \sum_{\la_j\in \ten} a(t;\la_j)E_jf_\la \bigl\|_{L^2_xL^p_t(\R\times M)}
\\
&\le \la \bigl\| \sum_{\la_j\in \ten} \, |D_t|^{1/2-1/p}a(t;\la_j)E_jf_\la \bigl\|_{L^2_xL^2_t(\R\times M)}.
\end{align*}
Since, by orthogonality
\begin{multline*} \bigl\| \sum_{\la_j\in \ten} \, |D_t|^{1/2-1/p}a(t;\la_j)E_jf_\la \bigl\|_{L^2_xL^2_t(\R\times M)}^2
\\
=   \sum_{\la_j\in \ten} \, \bigl\|
\, |D_t|^{1/2-1/p}a(t;\la_j)E_jf_\la \bigl\|_{L^2_xL^2_t(\R\times M)}^2,
\end{multline*}
we conclude that
\begin{multline}\label{d.7}
\|(I-\beta(-D_t/\la)w\|_{L^p_tL^q_x(\R\times M)}
\\
\lesssim \la \, \bigl(\sup_{\mu\in \ten} \bigl\| \, |D_t|^{1/2-1/p}a(t;\mu)\bigr\|_{L^2_t(\R)}\bigr) \cdot \|f_\la\|_{L^2(M)}.
\end{multline}

Next,  by Plancherel's theorem, \eqref{d.2} and \eqref{d.5},
\begin{align*}
\| \, |D_t|^{1/2-1/p}a(t;\mu)\|_{L^2_t(\R)}^2 &=(2\pi)^{-1} \int_{-\infty}^\infty
|\tau|^{1-2/p} \, \bigl|\Hat \eta(\tau + \mu^2/\la)\bigr|^2 \, \bigl|(1-\beta(-\tau/\la))\bigr|^2 \, d\tau
\\
&\lesssim \int_{\tau\notin [-5/\la/4,\, -3\la/4]} |\tau|^{1-2/p} \, |\Hat \eta(\tau+\mu^2/\la)|^2 \, d\tau.
\end{align*}
Note that $|\tau+\mu^2/\la|\approx (|\tau|+\la)$ if if $\tau\notin [-5\la/4,-3\la/4]$ and
$\mu\in\ten$, and since $\Hat \eta\in {\mathcal S}(\R)$ the preceding inequality leads to the trivial bounds
\begin{equation}\label{d.8}
\sup_{\mu\in \ten} \| \, |D_t|^{1/2-1/p}a(t;\mu)\|_{L^2_t(\R)} \lesssim \la^{-N}.
\end{equation}
Combining this inequality with \eqref{d.7} yields \eqref{d.3}.
\end{proof}

Using the lemma and the first estimate in Theorem~\ref{mainthm} it is very easy to prove 
\eqref{d.1'}.  We first note that we may apply this
Theorem, since if $w$ is as in \eqref{d.1'},
\begin{equation}
(i\la\partial_t+\Delta_g+i\la)w(t,x)
=\bigl(i\la \eta'(t)+i\la\eta(t)\bigr)\cdot
e^{it\la^{-1}\Delta_g}f_\la(x)
\quad \text{vanishes } \, \, \text{if }\, \,
t\notin [0,1],
\end{equation}
and $w(0,x)=0$.
Therefore by \eqref{i.13} and \eqref{d.3} we have
\begin{align}\label{d.11}
\|w&\|_{L^p_tL^q_x(\R\times M)}
\\
&\le \|\beta(-D_t/\la)w\|_{L^p_tL^q_x(\R\times M)}
+ \|(I-\beta(-D_t/\la))w\|_{L^p_tL^q_x(\R\times M)} \notag
\\
&\lesssim \la^{-1+1/p}
\bigl\| \, i\la(\eta'(t)+\eta(t)) \cdot
e^{it\la^{-1}\Delta_g}f_\la \, \bigr\|_{L^2(\R\times M)}
+\la^{-N}\|f_\la\|_{L^2(M)}
\notag
\\
&\lesssim 
\la^{1/p}\|f_\la\|_{L^2(M)},
\notag
\end{align}
as desired.

\medskip

Let us now prove dyadic high-frequency estimates
for $e^{-itH_V}$ where
\begin{equation}\label{d.12}
H_V=-\Delta_g+V
\end{equation}
with 
\begin{equation}\label{d.13}
V\in L^{n/2}(M) \, \, \text{if } \, 
\, n\ge 3, \, \, \text{and } \, \, 
V\in L^{1+\delta}(M), \, \, \, \text{some } \, 
\delta>0 \, \, \text{if } \, \, n=2.
\end{equation}
Let us focus first on the case where $n\ge3$ and then
handle $n=2$ later.

Under the assumption \eqref{d.13} $H_V$ defines a self-adjoint operator which is bounded from below.  
See e.g., \cite{BHSS}.  We wish to prove the analog
of \eqref{m.6} for the operators $e^{-itH_V}$.  If
necessary, we may add a constant to $V$ so that
\begin{equation}\label{d.14}
H_V\ge 0
\end{equation}
as we shall always assume.  This will not affect
our estimates, since, if we, say add the constant
$N$ to $V$ the two different Schr\"odinger operators
will agree up to a factor $e^{\pm itN}$.

Just as with the $V=0$ case, the eigenvalues of the operator $\sqrt{H_V}$ (defined by the spectral theorem)
are nonnegative, discrete and tend to infinity.  We
can list them counting multiplicity as 
$0\le \mu_1\le \mu_2\le \dots$, and there is an
associated orthonormal basis of eigenfunctions $\{e^V_j\}$
$$H_V e^V_j =\mu_j^2 e^V_j \quad \text{with }
\, \, \int_M |e^V_j|^2=1.$$
Analogous to the $V=0$ case, let $E^V_j$ denote the projection onto the $j$th eigenspace,
$$E^V_jf=\bigl(\int_M f\, \overline{e^V_j}\bigr) \cdot e^V_j.$$

Then for large $\lambda$ we wish to prove the analog of \eqref{m.6}:
\begin{equation}\label{d.15}
\bigl\|e^{-itH_V}f_\la\bigr\|_{L^p_tL^q_x([0,\la^{-1}]\times M)}
\le C\|f_\la\|_{L^2(M)} \quad
\text{if } \, \mathrm{spec } \,f_\la \in \ten,
\end{equation}
with the condition meaning that $E^V_jf_\la=0$ if
$\la_j\notin \ten$.  We are assuming the exponents
$(p,q)$ are as in \eqref{m.6}.  For later use, we note that since $e^{-itH_V}$ is a unitary operator
on $L^2(M)$ this estimate yields the unit-scale bounds
\begin{equation}\label{d.15'}\tag{3.15$'$}
\bigl\|e^{-itH_V}f_\la\bigr\|_{L^p_tL^q_x([0,1]\times M)}
\le C \la^{1/p} \, \|f_\la\|_{L^2(M)} \quad
\text{if } \, \mathrm{spec } \,f_\la \in \ten,
\end{equation}

Since, by the spectral theorem
$$\|e^{-itH_V}\|_{L^2(M)\to L^2(M)}=1,
$$
the estimate trivially holds for $p=\infty$ and
$q=2$.  Therefore, by interpolation, 
since we are currently assuming that $n\ge3$ it 
suffices to prove the estimate for the other endpoint,
i.e., that for $f_\la$ as in \eqref{d.15}
we have
\begin{equation}\label{d.16}
\bigl\|e^{-itH_V}f_\la\bigr\|_{L^2_tL^{2n/(n-2)}_x([0,\la^{-1}]\times M)}
\le C\|f_\la\|_{L^2(M)}.
\end{equation}
By scaling, this is equivalent to the statement
that, for $f_\la$ as above, we have
\begin{equation*}
\bigl\|e^{-it\la^{-1}H_V}f_\la\bigr\|_{L^2_tL^{2n/(n-2)}_x([0,1]\times M)}
\le C\la^{1/2}\|f_\la\|_{L^2(M)}.
\end{equation*}
Finally, as before, this is equivalent to showing that
whenever $$\eta\in C^\infty((0,1))$$ is fixed we have
\begin{equation} \label{d.17}
\bigl\|w\bigr\|_{L^2_tL^{2n/(n-2)}_x(\R\times M)}
\le C\la^{1/2}\|f_\la\|_{L^2(M)},
\quad \text{with } \, \, w(t,x)= \eta(t)\cdot 
e^{-it\la^{-1}H_V}f_\lambda,
\end{equation}
with $f_\lambda$ as above.

To proceed we need the analog of Lemma~\ref{error}.

\begin{lemma}\label{Verror}  Let $n\ge2$ and let $w$ be as in
\eqref{d.17} with $\eta\in C^\infty_0((0,1))$ and $f_\la$ as in \eqref{d.15}
and suppose that the exponents $(p,q)$ are as in
\eqref{m.6}.
Then for large enough $\la$ and each $N=1,2,\dots$
we have the uniform bounds
\begin{equation}\label{d.18}
\bigl\| \, (I-\beta(-D_t/\la))w \,
\bigr\|_{L^p_tL^q_x(\R \times M)}
\le C_N \la^{-N}\|f_\la\|_{L^2(M)}.
\end{equation}
\end{lemma}

Since, for instance, by \cite{BHSS} and  \cite{BSS}  we have the analog of
\eqref{d.s},
$$
\|u\|_{L^q(M)}\lesssim \|(I+H_V)^{1/2}u\|_{L^2(M)},
$$
for $q$ as in \eqref{m.6} it is clear that the proof of Lemma~\ref{error}
yields \eqref{d.18}.  For the two-dimensional case one uses the fact that, if,
as we are assuming $V\in L^{1+\delta}(M)$, $\delta>0$, then
$V$ is in the Kato class ${\mathcal K}(M)$.

\medskip

We now are positioned to prove \eqref{d.17}.  To take
advantage of our assumption \eqref{d.13} for a given
large $\ell>1$, as in \cite{BHSS}, let us split
$$V=V_{\le \ell}+V_{>\ell},
$$
where
$$V_{>\ell}(x)=V(x) \, \, \text{if } \, |V(x)|> \ell
\, \, \text{and } \, \, 0 \, \, \text{otherwise}.
$$
Our assumption \eqref{d.13} then yields
\begin{equation}\label{d.19}
\|V_{>\ell}\|_{L^{n/2}(M)}=\delta(\ell), \quad
\text{with } \, \, \delta(\ell)\to 0 \, \, 
\text{as } \, \ell\to \infty,
\end{equation}
and we also trivially have
\begin{equation}\label{d.20}
\|V_{\le \ell}\|_{L^\infty(M)}\le \ell.
\end{equation}

To use this we note that since $-H_V=\Delta_g-V$
\begin{align*}(i\la\partial_t+\Delta_g+i\la)w
&= (i\la\partial_t-H_V+i\la)w +Vw
\\
&= (i\la\partial_t-H_V+i\la)w +V_{\le \ell}\, w +V_{> \ell} \, w,
\end{align*}
and also $w(0,\cd)=0$.
So we can split
\begin{equation}\label{d.21}
w=\widetilde w+ w_{\le \ell} + w_{>\ell},
\end{equation}
where
\begin{equation}\label{d.22}
(i\la\partial_t+\Delta_g+i\la)\widetilde w
= (i\partial_t-H_V+i\la)w =\widetilde F, \quad 
\tilde w(0,\cd)=0,
\end{equation}
\begin{equation}\label{d.23}
(i\la\partial_t+\Delta_g+i\la)w_{\le \ell}= V_{\le \ell} \, w
=F_{\le \ell}, \quad w_{\le \ell}(0,\cd)=0,
\end{equation}
and
\begin{equation}\label{d.24}
(i\la\partial_t+\Delta_g+i\la)w_{> \ell}= V_{> \ell}\, w
=F_{> \ell}, \quad w_{> \ell}(0,\cd)=0,
\end{equation}
Note that since $w(t,x)=0$, $t\notin (0,1)$ each
of the forcing terms $\widetilde F$, $F_{\le \ell}$ 
and $F_{>\ell}$ also vanishes for such $t$ which allows
us to apply the estimates in Theorem~\ref{mainthm}
for $\widetilde w$, $w_{\le \ell}$ and $w_{>\ell}$.

By \eqref{d.18} and \eqref{d.21} we have for each $N=1,2,\dots$
\begin{align}\label{d.25}
\bigl\|w&\bigr\|_{L^2_tL^{2n/(n-2)}_x(\R\times M)}
\\
&\le \bigl\|\beta(-D_t/\la)w\bigr\|_{L^2_tL^{2n/(n-2)}_x(\R\times M)}+C_N \la^{-N}\|f\|_{L^2(M)} \notag
\\
&\le
\bigl\|\beta(-D_t/\la)\widetilde w\bigr\|_{L^2_tL^{2n/(n-2)}_x(\R\times M)}
+ \bigl\|\beta(-D_t/\la)w_{\le \ell}\bigr\|_{L^2_tL^{2n/(n-2)}_x(\R\times M)} \notag
\\
&+
\bigl\|\beta(-D_t/\la)w_{>\ell}\bigr\|_{L^2_tL^{2n/(n-2)}_x(\R\times M)}+ C_N \la^{-N}\|f\|_{L^2(M)}.
\notag
\end{align}
Based on this  we
would obtain \eqref{d.17} if we could show that
$\ell$ could be fixed large enough so that
we have the following three inequalities
\begin{equation}\label{d.26}
\bigl\| \beta(-D_t/\la) \widetilde w\bigr\|_{L^2_tL^{2n/(n-2)}_x(\R\times M)}
\le C\la^{1/2}\|f_\la\|_{L^2(M)},
\end{equation}
as well as
\begin{equation}\label{d.27}
\bigl\| \beta(-D_t/\la) w_{\le \ell}\bigr\|_{L^2_tL^{2n/(n-2)}_x(\R\times M)}
\le C\ell\la^{1/2}\|f_\la\|_{L^2(M)},
\end{equation}
and finally
\begin{equation}\label{d.28}
\bigl\| \beta(-D_t/\la) w_{> \ell}\bigr\|_{L^2_tL^{2n/(n-2)}_x(\R\times M)}
\le \tfrac12 \bigl\|w\bigr\|_{L^2_tL^{2n/(n-2)}_x(\R\times M)}.
\end{equation}
Indeed we just combine \eqref{d.25}--\eqref{d.28}
and use a simple bootstrapping argument which is justified since the right side of \eqref{d.28} is finite
by the aforementioned Sobolev estimates for $H_V$.

To prove these three estimates we shall use Theorem~\ref{mainthm},  as we may,  since,
as we mentioned before, the forcing terms in \eqref{m.22},
\eqref{m.23} and \eqref{m.24} obey the support assumption
in \eqref{m.14}.

To prove \eqref{i.13} we note that if $\widetilde F$
is as an \eqref{d.22} then, since $w$ is as in \eqref{d.17}, we have
$$\widetilde F(t,x)
= (i\partial_t-H_V+i\la) \bigl(\eta(t)e^{-it\la^{-1}H_V}f_\la(x)\bigr)
=i\la(\eta'(t)+\eta(t)) e^{-it\la^{-1}H_V}f_\la(x).$$
Consequently, as in the $V=0$ case considered before,
we may use the $L^2$-estimate, \eqref{i.13}, in
Theorem~\ref{mainthm} to deduce
that
\begin{align*}
\bigl\| \beta(-D_t/\la) \widetilde w\bigr\|_{L^2_tL^{2n/(n-2)}_x(\R\times M)}
&\le \la^{-1/2}\| i\la(\eta'(t)+\eta(t))\cdot e^{-it\la^{-1}H_V}f_\la\|_{L^2_{t,x}(\R\times M)}
\\
&\lesssim \la^{1/2} \|f_\la\|_{L^2(M)},
\end{align*}
as desired.
Similarly,  by \eqref{d.19} and 
the formula for $w$ in \eqref{d.17}, we obtain
\begin{align*}\bigl\| \beta(-D_t/\la)  w_{\le \ell}\bigr\|_{L^2_tL^{2n/(n-2)}_x(\R\times M)}
&\le C\la^{-1/2}\|V_{\le \ell} \, w\|_{L^2_{t,x}(\R\times M)}
\\
&\le C\ell \la^{-1/2}\| \eta(t) \cdot e^{-it\la^{-1}H_V}f_\la\|_{L^2_{t,x}(\R\times M)}
\\
&\le C'\ell \la^{-1/2}\|f_\la\|_{L^2(M)},
\end{align*}
which is better than the inequality posited in \eqref{d.27}.

Up until now we have not used the second inequality
in Theorem~\ref{mainthm}.  We need it to obtain
\eqref{d.28} which allows the bootstrapping step.
Note that
$$\frac1{q'}-\frac1q=\frac2n, \quad \text{if } \, \,
q=2n/(n-2), \,  \, q'=2n/(n+2).$$
Consequently if we use \eqref{i.14}, \eqref{d.24}, H\"older's inequality and \eqref{d.19} than
we conclude that we can
fix $\ell$ large enough so that we have
\begin{align*}\bigl\| \beta(-D_t/\la)  w_{> \ell}\bigr\|_{L^2_tL^{2n/(n-2)}_x(\R\times M)}
&\le C\|V_{>\ell} \, w\|_{L^2_tL^{2n/(n+2)}_x(\R\times M)}
\\
&\le C\|V_{>\ell}\|_{L^{n/2}(M)} \cdot 
\| w  \|_{L^2_tL^{2n/(n-2)}_x(\R\times M)}
\\
&\le \tfrac 12 \| w  \|_{L^2_tL^{2n/(n-2)}_x(\R\times M)},
\end{align*}
assuming, as we may, in the last step that, if $\delta(\ell)$ is as
in \eqref{d.19},  $C\delta(\ell)\le \tfrac12$.
Since this is the last of the three inequalities
we had to prove, we have established \eqref{d.17}
and hence \eqref{d.16}.

Next, let us point out that for functions only involving
low frequencies we have these types of estimates for unit time scales
for all of the exponents in \eqref{m.6} in all dimensions.
In other words if $C_0<\infty$ is fixed we claim that
\begin{equation}\label{d.29}
\bigl\|e^{-itH_V}f\bigr\|_{L^p_tL^{q}_x([0,1]\times M)}
\lesssim \|f\|_{L^2(M)}, \quad 
\text{if } \, \, \mathrm{spec } \, f
\subset [0,C_0].
\end{equation}
To see this we just note that by the Sobolev estimates
that were used in the proof of Lemma~\ref{Verror}
we have the following uniform bounds for all times $t$:
$$\| e^{-itH_V}f\|_{L^q(M)}
\lesssim \|\sqrt{I+H_V} \, e^{-itH_V}f\|_{L^2(M)}
\lesssim (1+C_0)\|f\|_{L^2(M)},
$$
by the spectral theorem
for $f$ as in \eqref{d.29}.

\medskip

Next, let us show that, for large enough $\lambda$,
when $n=2$ we have the estimates
in \eqref{d.15} for each fixed $(p,q)$ as in \eqref{m.6}.  Here
we can take advantage of the fact that we must have
$p>2$ and so the power of $\lambda$ in \eqref{i.14} is negative.  Since the bounds in \eqref{i.14} blow up
as the exponents in \eqref{d.15} approach the ``forbidden'' pair $(p,q)=(2,\infty)$
for $n=2$, one needs to choose $\lambda$ larger and
larger as $q$ increases.   On the other hand,
by interpolation, if we can establish
\eqref{i.14} for a given $q_0$ and large enough $\lambda$,
as before, by a trivial interpolation argument, we also
 obtain the bounds for all $q\in (2,q_0)$.
To take advantage of our assumption on the potential
in \eqref{d.13}, let us thus fix an exponent $q$
sufficiently large so that
\begin{equation}\label{d.30}\frac1{q'}-\frac1{q}\ge \frac1{1+\delta}.
\end{equation}
We then can just split $w$ into two terms, 
$w=\widetilde w+ w_V$,
one
being $\widetilde w$ exactly as before and the other
now solving 
$$(i\la\partial_t+\Delta_g+i\la)w_{V}= Vw
, \quad w_{V}(0,\cd)=0.$$
In other words, $w_V=w_{\le \ell}+w_{>\ell}$.

If we repeat the arguments for the $n\ge 3$ case we 
then deduce that we would have the estimates in
\eqref{d.15} for our exponents $(p,q)$ if 
\begin{equation*}
\bigl\| \beta(-D_t/\la) \widetilde w\bigr\|_{L^p_tL^{q}_x(\R\times M)}
\le C\la^{1/p}\|f_\la\|_{L^2(M)},
\end{equation*}
as well as
\begin{equation*}
\bigl\| \beta(-D_t/\la) w_{V}\bigr\|_{L^p_tL^{q}_x(\R\times M)}
\le \tfrac12 \bigl\|w\bigr\|_{L^p_tL^{q}_x(\R\times M)},
\end{equation*}
assuming that $\la$ is sufficiently large depending on $q$.

The first inequality follows from the argument used
before.  One just uses \eqref{i.13}.  

To prove
the second inequality we repeat the proof of 
\eqref{d.28}, noting that our assumptions on $q$
and $V$ ensure that, by H\"older's inequality $\|V\|_{L^r(M)}\le C_M \|V\|_{L^{1+\delta}(M)}<\infty$,
where $1/r=1/q'-1/q$, due to \eqref{d.30}.
As a result, if we use \eqref{i.14} and repeat
the proof of \eqref{d.28} we conclude that since
$w(t,x)=0$ for $t\notin [0,1]$
the left side of the second inequality is dominated
by 
\begin{multline*}\la^{-1+2/p}\|V w\|_{L^{p'}_tL^{q'}_x
([0,1]\times M)}
\le C_q \la^{-1+2/p}\|V\|_{L^r(M)} \, \|w\|_{L^{p'}_tL^{q}_x
([0,1]\times M)}
\\
\le C_{q,V} \la^{-1+2/p} \|w\|_{L^{p}_tL^{q}_x ([0,1]\times M)}
<\tfrac12 \|w\|_{L^p_tL^{q}_x},
\end{multline*}
for large enough $\lambda$ since $-1+2/p<0$.  In 
the second inequality we used H\"older's inequality
in the $t$ variable and the fact that $p>p'$.

\subsection*{Proof of Theorem~\ref{hvthm}}

Let us conclude the section by showing that the dyadic estimates that we have obtained can be used
along with Littlewood-Paley estimates associated with $H_V$ yield Theorem~\ref{hvthm}.  For
the sake of completeness, we shall give the simple proof of the Littlewood-Paley estimates involving
singular potentials in an appendix.

Let us state the estimates we require.  Recall that we are assuming as in \eqref{d.14}, as we may, that
$H_V\ge 0$, and so we may consider the operator $P_V=\sqrt{H_V}$.   If $\beta$ as in \eqref{i.7} and 
\eqref{i.8} is our Littlewood-Paley bump function, let
$$\beta_0(s)=1-\sum_{j=0}^\infty \beta(2^{-j}s)\in C^\infty_0([0,2)).$$
We shall then use the Littlewood-Paley estimates
\begin{equation}\label{d.31}
\|h\|_{L^q(M)}\lesssim \|\beta_0(H_V)h\|_{L^q(M)}\, +\, 
\bigl\| \, \bigl(\sum_{j=0}^\infty |\beta(  P_V/2^{j})h|^2\bigr)^{1/2}\, \bigr\|_{L^q(M)},
\end{equation}
provided that $V$ is as in Theorem~\ref{hvthm} and
\begin{equation}\label{d.32}
1<q<\infty \, \, \text{if } \, \, n=2,3,4 \quad \text{and } \, \, \tfrac{2n}{n+4}<q<\tfrac{2n}{n-4} \, \, \text{if } \, \, n\ge 5.
\end{equation}

We also note that since $\tfrac{2n}{n-2}<\tfrac{2n}{n-4}$ when $n\ge 5$, the exponents here include the exponents
$q$ arising in \eqref{i.2}.  Also, since $p\ge2$ and $q\ge 2$ if $(p,q)$ are as in \eqref{i.2}, we obtain from
\eqref{d.31} and Minkowski's inequality that we have for such exponents
\begin{multline}\label{d.33}
\| e^{-itH_V} f\|_{L^p_tL^q_x([0,1]\times M)}\lesssim \|\beta_0(H_V) e^{-itH_V} )f\|_{L^p_tL^q_x([0,1]\times M)}
\\
+\Bigl(\, \sum_{j=0}^\infty \bigl\|\beta( P_V/2^{j}) e^{-itH_V}f\bigr\|_{L^p_tL^q_x([0,1]\times M)}^2\, \Bigr)^{1/2}.
\end{multline}
Additionally, by \eqref{d.29} we have
\begin{equation}\label{d.34}
 \|\beta_0(H_V) e^{-itH_V} f\|_{L^p_tL^q_x([0,1]\times M)}
 \le \|\beta_0(H_V)f\|_{L^2(M)}.
 \end{equation}
 Similarly if we use \eqref{d.29} for small $j\ge 0$ and \eqref{d.15'} for large $j$ we obtain
\begin{equation}\label{d.35}
 \bigl\|\beta( P_V/2^{j}) e^{-itH_V}f\bigr\|_{L^p_tL^q_x([0,1]\times M)}  \le C_V 2^{j/p}
 \|\beta( P_V/2^{j})f\|_{L^2(M)}, \quad j=0,1,\dots,
 \end{equation}
 for some uniform constant $C_V<\infty$.
 
 If we recall \eqref{i.7} and \eqref{i.8} and combine \eqref{d.33}, \eqref{d.34} and \eqref{d.35} and use the spectral theorem
 we deduce that
 \begin{multline}\label{d.36}
 \| e^{-itH_V} f\|_{L^p_tL^q_x([0,1]\times M)} \lesssim \|\beta_0(H_V)f\|_{L^2(M)}
 +\Bigl(\, \sum_{j=0}^\infty \|2^{j/p}\beta( P_V/2^{j})f\|_{L^2(M)}^2\, \Bigr)^{1/2}
 \\
 \lesssim \|\, (\sqrt{I+H_V})^{1/p}f\|_{L^2(M)}.
 \end{multline}

This does not quite give us the estimate \eqref{i.6} in Theorem~\ref{hvthm}, since the right hand side
of this inequality involves the Sobolev space $H^{1/p}(M)$ defined as in \eqref{i.3} by the operator
$\sqrt{I-\Delta_g}$ as opposed to Sobolev space defined by the operator $\sqrt{I+H_V}$ as in \eqref{d.36}.
This, though, is easy to rectify.  By standard arguments (see e.g., the appendix in \cite{BHSS}
for the case where $n\ge3$ and the one in \cite{BSS} for the two-dimensional case), for the potentials
we are considering and for the exponents $q$ as above we have
$$\|\sqrt{I-\Delta_g} f\|_{L^2(M)} \approx \|\sqrt{I+H_V}f\|_{L^2(M)},$$
which means that the two $L^2$-Sobolev spaces of order $1$ are comparable.  By
interpolation this means that we have 
\begin{equation}\label{d.37}
\|(\sqrt{I-\Delta_g})^\sigma f\|_{L^2(M)} \approx \|(\sqrt{I+H_V})^\sigma f\|_{L^2(M)}, \quad 0\le \sigma \le 1,
\end{equation}
since the estimate for $\sigma=0$ is trivial.

If we combine \eqref{d.36} and \eqref{d.37}  we obtain \eqref{i.6}, which completes the proof of Theorem~\ref{hvthm}.

\newsection{Sharpness of the quasimode estimates}\label{sharpsec}

Let us now show that our scaled quasimode estimates
\eqref{i.13}
cannot be improved on any compact manifold $(M,g)$.  We
do this by a ``Knapp-type'' construction that is
adapted to our scaled Schr\"odinger operators.

First, recall that we can choose local coordinates vanishing at a given point $x_0\in M$ so that, in these
coordinates,
\begin{equation}\label{s.1}
\Delta_g = \partial_1^2 +\sum_{1<j,k\le n}
g^{jk}(x)\partial_j\partial_k + \sum_{k=1}^n
b_j(x)\partial_k.
\end{equation}
Here, $\partial_k=\partial/\partial x_k$, $k=1,
\dots, n$.  Here $(g^{jk})_{1<j,k\le n}$ is a smooth real
positive definite matrix, and the $b_j$ are also smooth real-valued functions.  See, e.g., \cite[Appendix C.5]{HorIII}.

Fix $a\in C^\infty_0((-1/10,1/10))$ which equals one 
near the origin, and set
\begin{equation}\label{s.2}
w(t,x)=e^{i\la(x_1-t)}a(\lambda;t,x),
\end{equation}
where
\begin{equation}\label{s.3}
a(\la;t,x)=a(x_1+2(t-1/2))\, a(\la^{1/2}(x_1-2t))\, 
a(\la^{1/2}|x'|), \quad
\text{with } \, \, x'=(x_2,\dots, x_n).
\end{equation}
Due to the exponential factor, the space-time Fourier transform of $w$ is the space-time Fourier transform
of $a(\la;t,x)$ translated by $(-\la,\la,0,0,\dots,0)$,
where the first-coordinate here is dual to the time coordinate
and the rest dual to the $x$-coordinates. Since the Fouier transform of $a(\la;x,t)$
is $O((1+\la^{-1/2}|(\tau,\xi)|)^{-N}$ for all $N$, 
and since $\beta(s)=1$ for $s\in [3/4,5/4]$,
it follows that
\begin{equation}\label{s.4}
(I-\beta(-D_t/\la))w=O(\la^{-N}), \quad \forall \, N.
\end{equation}
Also, for each fixed $t$ near $1/2$, $x\to a(\la;t,x)$ 
is one on a set of measure $\approx \la^{-n/2}$.  Thus,
by \eqref{s.3} and \eqref{s.4}, for sufficiently
large $\la$ we have
\begin{equation}\label{s.5}
\|\beta(-D_t/\la)w\|_{L^p_tL^q_x([0,1]\times M)}
\ge \|w\|_{L^p_tL^q_x([0,1]\times M)} -O(\la^{-N})
\ge c\la^{-n/2q}, \, \, \, \la \, \, \, \text{large },
\end{equation}
for some $c>0$.

Note also that
\begin{equation}\label{s.6}
F(t,x)=(i\la\partial_t+\Delta_g+i\la)w(t,x) =0, 
\quad t\notin (0,1), 
\quad \text{and } \, \, w(0,x)=0.
\end{equation}

Next, let us observe that
\begin{equation}\label{s.7}
\Bigl(\sum_{1<j,k\le n}g^{jk}(x)\partial_j\partial_k + \sum_{k=1}^n
b_j(x)\partial_k+i\la\Bigr)w =O(\la).
\end{equation}
Also, if we rewrite $a(\la;t,x)$ as
\begin{multline}\label{s.8}
a(\la;t,x)=a(\la^{1/2}(x_1-2t))\cdot
\tilde a(\la;t,x), 
\\
\text{where } \, \,
\tilde a(\la;t,x)= a(x_1+2(t-1/2)) \, 
a(\la^{1/2}|x'|),
\end{multline}
then 
\begin{equation}\label{s.9}
i\la\partial_t\tilde a(\la;x,t)=O(\la), \, \, 
\text{and } \, \, 
\partial_1^j\tilde a(\la;x,t)=O(1), \, \, j=1,2.
\end{equation}
Consequently, by Leibniz's rule we have
\begin{multline}\label{s.10}
\bigl(i\la \partial_t+\partial_1^2\bigr)w(t,x)
=a(\la;t,x)\cdot \bigl(i\la \partial_t+\partial_1^2\bigr)
e^{i\la(x_1-t)}
\\
+ e^{i\la(x_1-t)}\cdot
i\la\partial_t\bigl(a(\la^{1/2}(x_1-2t))\bigr)
+2\partial_1 \bigl(a(\la^{1/2}(x_1-2t))\bigr)
\cdot \partial_1 \bigl(e^{i\la(x_1-t)}\bigr)+O(\la).
\end{multline}
Note that the first term in the right vanishes, as
does the sum of the second and third terms.  

Therefore, by \eqref{s.1} and \eqref{s.7}--\eqref{s.10},
we conclude that
if $F$ is as in \eqref{s.6}
we have
\begin{equation*}
F=O(\la),
\end{equation*}
and since $F$ is supported on a set of measure
$\approx \la^{-n/2}$, we deduce that
\begin{equation}\label{s.11}
\|F\|_{L^2_{t,x}(\R\times M)}\le \la^{1-n/4}.
\end{equation}

If we combine \eqref{s.11} and \eqref{s.5} we deduce
that there must be a $c_0>0$ so that for sufficiently
large $\la$ we have
\begin{multline}\label{s.12}
\frac{\|\beta(-D_t/\la)w\|_{L^p_tL^q_x(\R\times M)}}
{\|F\|_{L^2_{t,x}(\R\times M)}}
\ge c_0 \, \la^{-1} \cdot \, 
\la^{n(\frac{1}4-\frac{1}{2q})}= c_0\, \la^{-1+1/p},
\\ \text{if } \, \, n(1/2-1/q)=2/p.
\end{multline}

By \eqref{s.6} and \eqref{s.12}, we deduce
that our $L^2$-quasimode estimate \eqref{i.13}
is saturated on {\em any compact manifold}.

\subsection*{Some remarks}

A challenging problem is to determine when the results of Burq,G\'erard and Tzvetkov~\cite{bgtmanifold} can be
improved, even just for the $V\equiv0$ case.  As they point out, the sharpness of $O(\la^{1/2})$ 
bounds for the 
$L^{2n/(n-2)}(S^n)$-norms of $L^2$-normalized spherical harmonics of the second author~\cite{sogge86} imply that, on the sphere,
the $L^2_tL^{2n/(n-2)}_x(S^n)$ Strichartz estimates \eqref{i.6} cannot be improved when $V\equiv0$.  On the other
hand, they were able to use results from \cite{sogge86} and the special nature of the Laplacian on the sphere to show
that for many cases besides this endpoint Strichartz estimate improved bounds hold here.

More dramatically, Bourgain and Demeter~\cite{BourgainDemeterDecouple} were able to show on 
the torus ${\mathbb T}^n$  for the case
where $q=p=2(n+2)/n$ and $V\equiv0$, the analog of \eqref{i.6} is valid with any Sobolev norm $H^\e({\mathbb T}^n)$
in the right.  It does not seem clear, though, how much improvements are possible here as one approaches the
endpoint case of $(p,q)=(2,2n/(n-2))$ beyond what holds just by interpolation with \eqref{i.6} for this exponent
and the dramatic improvements for $p=q=2(n+2)/n$.  It would also be interesting to determine which singular
potentials could be added so that $e^{-itH_V}$ enjoys similar bounds for the latter pair of exponents on tori.  Related
partial results for resolvent problems were obtained in our joint work with Blair and Sire~\cite{BHSS}.

It would also be very interesting to determine whether there is a wide class of manifolds (beyond just spheres and tori) for which
some of the estimates in \eqref{i.6} could be improved even for $V\equiv0$.  This seems to be a very challenging problem.
One avenue, which is suggested by the Knapp example above and recent work on eigenfunctions (e.g., \cite{BlairSoggeToponogov}, \cite{SBLog}, \cite{SoggeKaknik}, \cite{sogge2015improved} and
 \cite{SoggeZelditchL4} ) might be to try to prove ``Kakeya-Nikodym'' estimates that link Strichartz estimates to ones
 involving products of powers of $L^2(M)$ norms and powers of supremums of  $L^2$-norms over shrinking tubes.
 
 The construction above suggests (not unexpectedly) that the tubes should be $\lambda^{-1/2}$-neighborhoods
 of the projection onto $(t,x)$ space of integral curves of the Hamilton flow of
 $$p(x,t, \tau,\xi)=\tau + Q(x,\xi),$$
 where $Q(x,\xi)$ is the principal symbol of the $-\Delta_g$, which in local coordinates
 is given by
 $$Q(x,\xi)=\sum_{j,k=1}^n g^{jk}(x)\xi_j\xi_k.$$
 Here $g^{jk}(x)$ denotes the cometric.  
 
 To be more specific, one might expect to control high-frequency solutions
 of Schr\"odinger equations by ``Kakeya-Nikodym norms'' over shrinking tubes about curves which in local coordinates
 are of the form
 $\gamma(t)=(t_0+t, x_0+x(t))$  where $x(t)$ is a geodesic with $\dot x(0)=\frac{\partial Q}{\partial \xi}(x_0,\xi_0)$ and
 speed $2Q(x_0,t_0)$, with $(t_0,x_0)\in \R\times M$.
 
 This approach proved to be successful even for ``critical norms'' in the related case of 
 estimates for eigenfunctions and in the aforementioned works improved eigenfunction estimates 
 versus the universal bounds \cite{sogge88} of one of us were obtained for manifolds of nonpositive curvature.  It would
 be very interesting to prove a corresponding result for high-frequency solutions of the unperturbed Schr\"odinger equation.

\newsection{Appendix: Littlewood-Paley and multiplier bounds involving $L^{n/2}$-potentials}\label{appendix}

Consider a nonnegative self-adjoint operator $H_V=-\Delta_g+V$ on a compact manifold.  Consider also
a Mikhlin-type multiplier $m\in C^\infty(\R_+)$, meaning that
\begin{equation}\label{h.1}
|\partial_\tau^j m(\tau)|\le C(1+\tau)^{-j}, \tau >0, \, \, \,
0\le j\le n/2 +1.
\end{equation}

We shall also assume that we have finite propagation speed for the wave equation associated to $H_V$.  By this
we mean that if $u,v\in L^2(M)$ and $d_g(\text{supp }u,\text{supp }v)=R$ then
\begin{equation}\label{h.2}
\bigl(u, \, \cos t\sqrt{H_V} \, v\bigr)=0, \, \, |t|<R.
\end{equation}

Additionally for a given $2<q_0<\infty$ we shall assume
that one has the  Bernstein (dyadic Sobolev)
estimates
\begin{multline}\label{h.3}
\|\beta(\sqrt{H_V}/\la) u\|_{L^{q_0}(M)}
\le C\la^{n(\frac12-\frac1{q_0})}\|u\|_{L^2(M)}, \, \,
\la \ge 1,
\\
\text{and } \, \,
\|\beta_0(\sqrt{H_V})u\|_{L^{q_0}(M)}\le C\|u\|_{L^2(M)},
\, \, \text{if } \, \beta_0(s)=1-\sum_{k=1}^\infty
\beta(2^{-k}s).
\end{multline}
Note that we would automatically have these bounds if
we had the natural  heat estimates of
Li and Yau~\cite{LiYau} for small times (see \cite{BSS}).

Using a result of Blunck \cite[Theorem 1.1]{Blunck}
we claim that we can obtain the following.  

\begin{theorem}\label{mult}
Assume that \eqref{h.2} and \eqref{h.3} are valid and
that $m$ is as in \eqref{h.1}.  Then
\begin{equation}\label{multineq}
\bigl\|m(\sqrt{H_V})f\bigr\|_{L^q(M)}
\le C_q\|f\|_{L^q(M)} \quad \forall \, q\in 
(q_0',q_0).
\end{equation}
\end{theorem}

If the standard small time pointwise heat kernel bounds  held, then
the the results in Theorem~\ref{mult} hold for all $1<q<\infty$  by Alexopoulos~\cite{AlexSP}.
However, as was shown in \cite{BSS}, following \cite{AZ} and  \cite{Simonsurvey}, the standard
small time heat kernel estimates need not hold if $V\in L^{n/2}(M)$, since there
can be unbounded eigenfunctions.  On the other hand, assuming that $V\in {\cal K}(M)$
ensures that these estimates hold by Sturm~\cite{Sturm}.
Here, ${\mathcal K}(M)$ denotes the Kato class.  Recall also that if $V\in L^{n/2+\delta}(M)$, $\delta>0$, then $V\in {\mathcal K}(M)$, and
that $L^{n/2}(M)$ and ${\mathcal K}(M)$ enjoy the same scaling properties.

By the results in \cite{BHSS}, \eqref{h.3} holds for
$V\in L^{n/2}(M)$ if $2<q_0<\infty$ if $n=3,4$, and
if $q_0=2n/(n-4)$ if $n\ge 5$.  By results in 
\cite{BSS}, if $n=2$ one also has this bound for
all $2<q_0<\infty$ if $V\in {\mathcal K}(M)$.  One obtains these dyadic Sobolev
estimates in higher dimensions $n\ge5$ directly from Sobolev estimates proved in
\cite{BHSS} and for $n=2,3,4$ by a simple orthogonality argument and the quasimode estimates
proved in \cite{BHSS} and \cite{BSS} in the  other cases.
By a result of Coulhon and Sikora \cite{CouS},
\eqref{h.2} is valid when $H_V$ is nonnegative,
self-adjoint and $V\in L^1(M)$.  Alternately, one can use arguments from \cite{BSS} to show this for the potentials that
we are considering.

Consequently, the estimates \eqref{multineq} are valid for the potentials we are considering
provided that $1<q<\infty$ if $n=2,3,4$ and for $\tfrac{2n}{n+4}<q<\tfrac{2n}{n-4}$ if $n\ge 5$.\footnote{We should point out
that these results also are a consequence of estimates in \cite{SikoraMultipliers}, and, in fact, a stronger theorem involving
weaker regularity assumption on the multiplier $m$ also holds.  On the other hand, since the proof of Theorem~\ref{mult} is simple, and since it easy
to use the main Theorem in \cite{Blunck} to see the ingredients that are needed, we have chosen to include the proof here for the sake of completeness.
We also do this since checking that the  hypotheses for the very general results  in \cite{SikoraMultipliers} which are needed to obtain Theorem~\ref{mult}  is a bit laborious.}
 Therefore, by a standard argument involving Radamacher functions (see e.g., \cite[p. 21]{SFIO2})
we obtain the Littlewood-Paley estimates \eqref{d.31} that we used at the end of \S~\ref{stsec}:

\begin{corr}\label{LPcorr}  If $V\in L^{n/2}(M)$
and $1<q<\infty$ for $n=3,4$ or $2n/(n+4)<q<2n/(n-4)$
for $n\ge 5$ then whenever \eqref{h.1} is valid we
have $m(\sqrt{H_V}): \, L^q(M)\to L^q(M)$.  If 
$n=2$, $V\in {\mathcal K}(M)$ and $1<q<\infty$ then
these bounds also hold.  Consequently, under
these hypotheses we have the Littlewood-Paley estimates
\begin{equation}\label{lpest}
 \|h\|_{L^q(M)} \le C_{q,V}\, 
\|\beta_0(\sqrt{H_V})h\|_{L^q(M)}
+
\bigl\| \, \bigl(\, \sum_{k=1}^\infty
\, \bigl|\beta(\sqrt{H_V}/2^k)h\bigr|^2\, \bigr)^{1/2} \,
\bigr\|_{L^q(M)},
\end{equation}
for $q$ and $V$ as above.
\end{corr}

\begin{proof}[Proof of Theorem~\ref{mult}]
By Theorem 1.1 in Blunck~\cite{Blunck}, it suffices
to show that, if $q\in (2,q_0)$, we have for some $\e_q>0$
\begin{equation}\label{bl}
\Bigl\| \, \1_{B(x_0,r)} \, e^{-\frac{r^2}2H_V} \, \1_{B(y_0,r)}\,
\Bigr\|_{L^{q'}(M) \to L^q(M)}\lesssim
r^{-n(\frac1{q'}-\frac1q)} e^{-\e_q \frac{d_g(x_0,y_0)}r}.
\end{equation}  
Here $\1_{B(x_0,r)}$ is the operator which is multiplication by the indicator function of 
the geodesic ball $B(x_0,r)$ of radius $r$ centered at 
$x_0$.  We may assume that $r$ is small, say,
smaller than half the injectivity radius, since otherwise
the estimate is trivial due to \eqref{h.3} and a simple $TT^*$ argument.

Let us first use \eqref{h.2} to deduce that
\begin{equation}\label{bl2}
\Bigl\| \, \1_{B(x_0,r)} \, e^{-\frac{r^2}2H_V} \, \1_{B(y_0,r)}\,
\Bigr\|_{L^{2}(M) \to L^2(M)}\lesssim
e^{-cd_g(x_0,y_0)/r},
\end{equation}
for some $c>0$.  We should note that, like \eqref{h.3},
\eqref{bl2}
automatically holds when one has the standard small-time heat kernel estimates.

We may assume that $d_g(x_0,y_0)\ge 10 r$, since otherwise
the result is trivial.  In this case, choose
$\rho\in C^\infty_0((-1/2,1/2))$ with $\rho(s)=1$ near
the origin.  Then since
$$e^{-\frac{r^2}2 H_V}
=\frac1{\sqrt{2\pi}}\int 
\frac1r e^{-\frac12 (t/r)^2}\,    \cos t\sqrt{H_V}\, dt,
$$
and by \eqref{h.2}
$$\1_{B(x_0,r)}\, \cos t\sqrt{H_V} \, 
\1_{B(y_0,r)} =0 \quad \text{if } \, \, t<R_0=d_g(x_0,y_0)/2,
$$
we must have
\begin{multline*}\1_{B(x_0,r)} \, e^{-\frac{r^2}2 H_V}\,  \1_{B(y_0,r)}
\\
=  \1_{B(x_0,r)} \,
\Bigl(\, \frac1{\sqrt{2\pi}}\int
(1-\rho(t/R_0))\, 
\frac1r e^{-\frac12 (t/r)^2}\, \,  \cos t\sqrt{H_V} \, dt\, \Bigr)
\, \1_{B(y_0,r)}.\end{multline*}
Consequently,
$$\int
|(1-\rho(t/R_0))|\, 
\frac1r e^{-\frac12 (t/r)^2}\, dt \lesssim
e^{-c d_g(x_0,y_0)/r}, \quad R_0=d_g(x_0,y_0)/2,
$$
and so, by the spectral theorem, 
\begin{multline*}
\Bigl\| \, \1_{B(x_0,r)} \, e^{-\frac{r^2}2H_V} \, \1_{B(y_0,r)}\,
\Bigr\|_{L^{2}(M) \to L^2(M)}
\\
\lesssim
\Bigl\|
\frac1{\sqrt{2\pi}}\int (1-\rho(t/R_0))\, 
\frac1r e^{-\frac12 (t/r)^2}\,    \cos t\sqrt{H_V} \, dt \,
\Bigr\|_{L^2(M)\to L^2(M)}
\lesssim e^{-cd_g(x_0,y_0)/r},
\end{multline*}
as claimed.

Next, by \eqref{h.3}
\begin{align*}
\bigl\| e^{-\frac{r^2}4H_V}f\bigr\|_{L^{q_0}(M)}
&\le \|\beta_0(\sqrt{H_V}) e^{-\frac{r^2}4H_V}f\|_{L^{q_0}(M)}
+\sum_{k=1}^\infty \bigl\|\beta(\sqrt{H_V}/2^k) e^{-\frac{r^2}4H_V}f \bigr\|_{L^{q_0}(M)}
\\
&\lesssim
\|f\|_{L^2(M)}+\sum_{k=1}^\infty
2^{nk(\frac12-\frac1{q_0})} \bigl\|\beta(\sqrt{H_V}/2^k) e^{-\frac{r^2}4H_V}f \bigr\|_{L^2(M)}
\\
&\lesssim 
\Bigl(\, 1\, +\, \sum_{k=1}^\infty 2^{nk(\frac12-\frac1{q_0})}
e^{-\frac14 (r2^k)^2}\, \Bigr) \, \|f\|_{L^2(M)}
\\
&\lesssim r^{-n(\frac12-\frac1{q_0})}\|f\|_{L^2(M)}.
\end{align*}

By a $TT^*$ argument this yields
\begin{equation}\label{bl3}
\bigl\| e^{-\frac{r^2}2H_V}\bigr\|_{L^{q_0'}(M)
\to L^{q_0}(M)}\lesssim r^{-n(\frac1{q_0'}-\frac1{q_0})}.
\end{equation}
By the M. Riesz interpolation theorem, \eqref{bl2}
and \eqref{bl3} yield \eqref{bl} for all
$q\in (2,q_0)$, as desired.
\end{proof}

%\setcitestyle{numbers} % set the citation style to ``numbers''.
\bibliography{refs}
\bibliographystyle{abbrv}
\end{document}